\documentclass[12pt]{amsart}
\usepackage{latexsym,fancyhdr,amssymb,color,amsmath,amsthm,graphicx,listings,comment}
\newtheorem{thm}{Theorem}       
\newtheorem{lemma}{Lemma}       \newtheorem{coro}{Corollary}
\usepackage[section]{placeins}
\let\paragraph\subsection
\title{On a theorem of Grove and Searle}
\author{Oliver Knill}
\date{06/21/2020}
\address{Department of Mathematics \\ Harvard University \\ Cambridge, MA, 02138 }

\begin{document}

\begin{abstract}
A theorem of Grove and Searle directly establishes that positive curvature $2d$ 
manifolds $M$ with circular symmetry group of dimension $2d \leq 8$ have positive Euler 
characteristic $\chi(M)$: the fixed point set $N$ consists of even dimensional positive curvature
manifolds and has the Euler characteristic $\chi(N)=\chi(M)$. It is not empty by Berger. If
$N$ has a co-dimension $2$ component, Grove-Searle forces $M$ to be in
$\{ \mathbb{RP}^{2d},\mathbb{S}^{2d},\mathbb{CP}^d \}$. By Frankel, there can be not two codimension $2$ cases.
In the remaining cases, Gauss-Bonnet-Chern forces all to have positive Euler characteristic. 
This simple proof does not quite reach the record $2d \leq 10$ which uses methods of Wilking
but it motivates to analyze the structure of fixed point components $N$ and
in particular to look at positive curvature manifolds which admit a $U(1)$ or $SU(2)$ symmetry with
connected or almost connected fixed point set $N$. They have amazing geodesic properties: 
the fixed point manifold $N$ agrees with the caustic of each of its points and the geodesic flow is 
integrable. In full generality, the Lefschetz fixed point property $\chi(N)=\chi(M)$ and Frankel's dimension theorem
${\rm dim}(M)< {\rm dim}(N_k) + {\rm dim}(N_l)$ for two different connectivity components of $N$
produce already heavy constraints in building up $M$ from smaller components. It is 
possible that $\mathbb{S}^{2d},\mathbb{RP}^{2d},\mathbb{CP}^{d},\mathbb{HP}^{d},\mathbb{OP}^2,
W^6,E^6,W^{12},W^{24}$ are actually a complete list of even-dimensional positive curvature manifolds admitting a 
continuum symmetry. Aside from the projective spaces, the Euler characteristic of the known cases
is always $1,2$ or $6$, where the jump from $2$ to $6$ happened with the Wallach or Eschenburg manifolds 
$W^6,E^6$ which have four fixed point components $N=S^2 + S^2 + S^0$, the only known case which 
are not of the Grove-Searle form $N=N_1$ or $N=N_1 + \{p\}$ with connected $N_1$. 
\end{abstract}
\maketitle

\section{Positive curvature manifolds}

\paragraph{}
$\mathcal{M}$ is the class of compact Riemannian manifolds admitting
a {\bf positive curvature} metric. $\mathcal{K}$ denotes the subset of $\mathcal{M}$ which allow a
metric containing $G=\mathbb{T}^1$ in the {\bf isometry group} and $\mathcal{P} = \{ M \in \mathcal{M}, \chi(M)>0 \}$ 
are the manifolds in $\mathcal{M}$ of {\bf positive Euler characteristic}. The
{\bf Hopf conjecture} $\mathcal{H}_{2d} \subset \mathcal{P}$
for all $d \geq 0$, is open in dimension $6$ or higher. The weaker statement 
$\mathcal{K}_{2d} \subset \mathcal{P}$ is currently known to be true in dimensions 
$\leq 10$ \cite{PuettmannSearle,RongSu2005},and relies on work of Grove-Searle \cite{GroveSearle}, 
and Wilking \cite{Wilking2003}. For cosmetic reasons we include
$0$-manifolds in $\mathcal{M}$ and $1$-manifolds into $\mathcal{M},\mathcal{K}$.

\paragraph{}
Not a single element in $\mathcal{M} \setminus \mathcal{K}$ is known. 
In six dimensions, $\{\mathbb{S}^{6},\mathbb{P}^{6},
\mathbb{CP}^3,W^6,E^6 \}$ represent all known connected topology 
types in $\mathcal{M}_6$, where $W^6,E^6$ are the Wallach \cite{Wallach1972} 
and Eschenberg manifolds \cite{Eschenburg1982}.
All these examples are also in $\mathcal{K}_6$. The full classification of 
$\mathcal{K}_4 = \{ \mathbb{S}^4,\mathbb{P}^4,\mathbb{CP}^2\}$ was given by
Hsiang-Kleiner \cite{HsiangKleiner}. 
The statement $\mathcal{K}_{2d} \subset \mathcal{P}$ appears
for $2d=6$ in \cite{Petersen3} (2. Edition, Cor. 8.3.3), the cases $2d=8$ is also covered by
\cite{PuettmannSearle}, the case $2d=10$
was established by Rong-Su \cite{RongSu2005}. 
For more on the subject of positive curvature manifolds see \cite{Ziller,WilkingZiller,
Kennard2013,AmannKennard, Grove2017,Berger2002} for overviews.

\paragraph{}
For $M \in \mathcal{K}$, the {\bf fixed point set} $N=\phi(M)$ of the isometry group is a union 
of totally geodesic sub-manifolds, each having even co-dimension in $\mathcal{M}$. (For $3$-manifolds, 
$N$ can be a circle bouquet which has the circle as the universal cover.)
The Lefshetz fixed point theorem implies that the {\bf Conner-Kobayashi map} $\phi$ \cite{Conner1957,Kobayashi1958}
satisfies $\chi(\phi(N)) = \chi(M)$. The manifold $N=\phi(M)$ can have components of mixed dimension 
but each is a positive curvature manifold sitting in $M$ and is geodesic in the sense 
that a geodesics in $N$ is also a geodesic in $M$. In particular, components can consist of points.
See \cite{Kobayashi1972} or \cite{Petersen3} Chapter 8. We focus in the following mostly on even dimensions
as there the big open Hopf problem is open in dimension $d$ larger than $10$ in the case with symmetries.

\section{The Grove Searle Theorem}

\paragraph{}
The {\bf Grove and Searle theorem} is that if $\phi(N)$ is connected of co-dimension $2$,
then $M$ is a {\bf spherical space-form}: $\mathcal{M}_{GS} = \{ \mathbb{S}^n, 
\mathbb{RP}^n,\mathbb{S}^{n}/\mathbb{Z}_m \}$. In even dimensions, these are the spheres 
$\mathbb{S}^{2d}$ or real projective spaces $\mathbb{RP}^{2d}$. 
In odd dimensions, one has $M=\mathbb{S}^{2d+1}/\mathbb{Z}_m$  with $\pi_1(M)=\mathbb{Z}_m$.
They are manifolds for $d>0$ and circle bouquet orbifolds $\mathbb{S}^1/\mathbb{Z}_m$ for $d=0$. 

\paragraph{}
Call $\phi(M)$ {\bf almost connected} if $\phi(M) = N \cup \{p\}$ with ${\rm dim}(N)>0$.
As ${\rm dim}(M)-{\rm dim}(N_k)$ is even for all components $N_k$ of $N$, 
the manifolds $N,M$ must be even-dimensional then.
The theorem of \cite{GroveSearle} is:

\begin{thm}[Grove Searle]
Assume $N$ has codimension $2$ in $M \in \mathcal{K}$. 
If $N=\phi(M)$ connected or $N=\mathbb{S}^0$ then $M$ is a spherical space form. 
If $M \in \mathcal{K}$ and $N$ is almost connected, then $M=\mathbb{CP}^d$.
\label{Theorem1}
\end{thm}

\paragraph{}
The almost connected case is also of interest in {\bf K\"ahler geometry}:
$\mathbb{CP}^{d}$ admits $\mathbb{T}^1$ actions with $N=\mathbb{CP}^{d-1} \cup \{p\}$.
For the complex projective plane $d=2$ for example, we can work with equivalence classes of complex lines in 
$\mathbb{C}^3$ and have $T_s(z_1,z_2,z_3) = (e^{is} z_1, e^{is} z_2,z_3\}$ which fixes the projective line
$\mathbb{CP}^1 = \{ (z_1,z_2,0) \}$ at infinity and the point $p=(0,0,1)$. For K\"ahler manifolds, where
$\mathcal{K}_{{\bf Kaehler}}=\mathcal{M}_{{\bf Kaehler}}$ is even more likely,
it is known that $\mathcal{K}_{6,{\bf Kaehler}} = \{ \mathbb{CP}^3 \}$ \cite{Johnson1984}. 
The statement $\mathcal{K}_{2d,{\bf Kaehler}} = \{ \mathbb{CP}^d \}$ is a conjecture of Frankel \cite{Frankel1961}.

\paragraph{}
Besides the spheres $\mathbb{S}^{2d}$, the infinite families of real,complex or quaternionic
projective spaces $\mathbb{RP}^{2d},\mathbb{CP}^{d},\mathbb{HP}^d$, only
$\mathbb{W}^6,\mathbb{E}^6,\mathbb{W}^{12},\mathbb{OP}^2 \in \mathcal{M}_{16}$ 
and $\mathbb{W}^{24}$ are known in even dimensions \cite{Ziller2}. The {\bf octionic projective plane} $\mathbb{OP}^2$
is also called Moufang-Cayley plane named after Ruth Moufang who introduced it in 1933 and
Arthur Cayley who introduced octonions in 1845. 
It is remarkable that the list $\mathbb{S}^{8}=\mathbb{OP}^1, \mathbb{OP}^2$ stops there. 
There is no $3$-dimensional (over $\mathcal{O}$) octonion projective space. 
All these positive curvature manifolds have circular symmetry. 
See \cite{Shankar1999} for lists of isometry groups. 

\paragraph{}
Besides the circle $U(1)$ which is the unit sphere in $\mathbb{C}$, the unit sphere $SU(2)$
in $\mathbb{H}$ is a natural candidate for a symmetry as this is the only Lie group which is
a sphere and so has positive curvature. The analogue question is then which manifolds in $\mathcal{M}$
allow for a metric such that there is an effective $SU(2)$ action for which the fixed point set
is connected of co-dimension $4$ or almost connected of co-dimension $4$. 
The connected case seems to happen only for $\mathbb{S}^{2d} \to \mathbb{S}^{2d+4}$ extensions
which can also be done using two $U(1)$ extensions. The almost connected case
appears with $\mathbb{HP}^d \to \mathbb{HP}^{d+1}$. 

\section{Positive Euler characteristic in dimension 8 or less}

\paragraph{} 
Theorem \ref{Theorem1} immediately implies

\begin{coro}
For $2d \leq 8$, we have $\mathcal{K}_{2d} \subset \mathcal{P}$.
\end{coro}

\begin{proof}
Every connectivity component $N_k$ of $N=\phi(M)$ 
with ${\rm dim}(N_k)\leq 4$ has $\chi(N_k)>0$ by of Gauss-Bonnet-Chern \cite{Chern1966},
a remark by Milnor. By a theorem of Frankel, if there are two components, then the 
sum of their dimensions is smaller than the dimension of $M$ (\cite{Frankel1961} Theorem 1 p 169).
If there is a co-dimension $2$ component, then by Grove Searle, $M$ must be a sphere or 
real projective space so that $M$ has positive Euler characteristic. If all components
have components smaller than $6$ dimension, we don't even need a Frankel type result and
have all components of positive Euler characteristic because the Gauss-Bonnet-Chern theorem
in $2d=4$ or Gauss-Bonnet in $2d=2$ kicks in and points automatically have positive Euler
characteristic. 
\end{proof}

\paragraph{}
In dimension $2d=10$, there could be a $6$-dimensional component and nothing else
and now, neither Gauss-Bonnet nor Grove-Searle catches in. I think there should
still be a Grove-Searle result in codimenion $4$ but then include $\mathbb{HP}^d$
in the list. That would lead to a Grove-Searle type proof for $2d=10$. 
For now, the currently best result in dimension $10$ use also cohomology theorems 
of Wilking \cite{Wilking2003} which are sophisticated and of independent interest.
\footnote{Thanks to Wolfgang Ziller for the references, especially the work
of Grove-Searle and Wilking and explaining these cases and to Burkhard Wilking for
explaining a trickier part in the case $2d=10$. }

\paragraph{}
We are fascinated by the Grove-Searle theorem because it is an elegant approach
to prove positive Euler characteristic in cases of symmetry, at least for small
dimensions. I have explored then the question whether the codimension 2 assumption
can be replaced with higher co-dimension case for some time believed to be able to prove that one
a connected $N$ forces $M$ to be in $\mathbb{RP}^d$, $\mathbb{CP}^d$ or $\mathbb{S^d}$.
\footnote{Thanks also to Karsten Grove, Catherine Searle and Lee Kennard to 
look and correct rather wild guesses of mine in this context.}

\paragraph{}
Here is the example, the experts immediately pointed out: let $M=\mathbb{HP}^d$ be the
quaternionic projective space, which is a compact $4d$-dimensional manifold
of positive curvature and Euler characteristic $d+1$. The manifold consists
of all lines $s \to [z_1,z_2, \dots, z_d] s $ with quaternionic components $z_k$ and
quaternion time $s \in \mathbb{H}$. It admits the circular action $z \to e^{it} z$
of the circle where the multiplication is from the left. The fixed points are
exactly all the tuples $[z_1,z_2,\dots, z_d] s$ with complex components $z_k$,
meaning that $N=\mathbb{CP}^d$ is the fixed point set. It is connected.

\paragraph{}
The example proves also directly that $\chi(\mathbb{HP}^d) = \chi{\mathbb{CP}^d})$. 
By the way, one can also compute the Euler characteristic of $\mathbb{CP}^d$ like that: 
take the circle action $[z_1,z_2,\dots, z_d] \to [z_1,z_2,\dots, e^{it} z_d]$ which has the
fixed point set $\mathbb{CP}^{d-1}$ as well as a point $p=[0,0,0,\dots,0,1]$. Induction
then gives $\chi(\mathbb{CP}^d) = d+1$. We see that the Lefschetz
fixed point principle $\chi(M) = \chi(N)$ is handy to compute Euler characteristic.

\paragraph{}
Just to dwell a bit more on the Lefschetz fixed point principle:
in the case $\mathbb{RP}^{2d}$ which consists of equivalence classes $[x_0,x_1,\dots, x_{d}]$
one has a circle action 
$$t \to [ x_0, \cos(t) x_1, \sin(t) x_2, \cos(t) x_3, \sin(t) x_4,\dots \cos(t) x_{d-1}, \sin(t) x_d] $$
which has the fixed point $[1,0,0,\dots,0]$ only. The Euler characteristic is $1$. 
In the case $\mathbb{RP}^{2d}$ one can look at the same circle action which then has the 
fixed points $[1,0,0,\dots,0]$ and $[-1,0,\dots,0]$. The Euler characteristic is $2$. 

\paragraph{}
The manifolds $\mathbb{HP}^d$ can also be constructed using Grove-Searle expansion.
It uses the group $SU(2)$. It is completely analog to how
the extension $\mathbb{CP}^d \to \mathbb{CP}^{d+1}$ is done using a circular
action which leaves $\mathbb{CP}^d$ as well as a point invariant. 
If the same thing is done just using the unit sphere $SU(2)$ in the quaternions
rather than the unit sphere $U(1)$ in the complex plane, we have the same
Grove-Searle picture also for $\mathbb{HP}^d$. What was for me not visible at first is
that there is a circular extension from $\mathbb{CP}^d$ to $\mathbb{HP}^d$ which has
the connected fixed point set. It is possible since the orbits of the extensions
are then $2d$ spheres.

\section{A constructive approach to Grove Searle}

\paragraph{} 
In the Grove-Searle case of a co-dimension-$2$ component, it is possible to 
{\bf reconstruct} $M$ from $N$. It uses only properties of the {\bf geodesic flow}.
By Hopf-Rynov, $M$ is complete so that for $x \in M, v \in T_xM, |v|=1$, there is a 
{\bf geodesic path} $\gamma_t(x,v) \in M$. The {\bf cut locus} of $x$ is set $y \in M$ 
for which multiple minimizing geodesics connecting $x$ and $y$ exist.
The set of {\bf conjugate points} of $x$ is the set of $y$ such that a parametrized family
of geodesics connecting $x$ and $y$ exists.

\paragraph{}
{\bf Definition.} For $M \in \mathcal{K}$ with fixed point set $N$, a {\bf geodesic} $\gamma(x,v)$ with $x \in N$
either stays in $N$ or then leaves $N$ at $x$. If it returns to $N$ at $x'$, it defines a $2$-dimensional 
{\bf sphere tube} $\Gamma(x,v)=G\gamma(x,v)$ as the orbit of $\gamma(x,v)$ under the circle $G$ action. It
can not have a self intersection. For $v \to T_xN$ the tube degenerates to a geodesic path on $N$.
If the {\bf return time} $T(x,v)$ to $N$ is positive, then it is larger or equal than the 
{\bf radius of injectivity} of $M$ and smaller or equal than the {\bf diameter} of $M$.

\begin{lemma} 
For every $v$ in the tangent space $T_xM$ away from $T_xN$, the tube 
$\Gamma(x,v)$ is a $2$-sphere consisting of geodesics that start and end
at $N$ in finite time $T(x,v)$.
\end{lemma}

\begin{proof}
$\Gamma(x,v)$ is a ray of geodesics starting at $x$. In a positive curvature manifold it
has to close at a point $x'$ where the rays focus is a conjugate point of $x$.
We have $x' \in N$ because the radius of the sphere $C(z)$ is positive for every $z \neq N$. 
The tube $\Gamma(x,v)$ is the orbit $G \gamma(x,v) = \{ T_s \gamma(x,v), s \in G\}$ 
of the geodesic $\gamma(x,v)$ under the $G=\mathbb{T}^1$ action.
For $z \in \Gamma(x,v)$, there is unique geodesic segment $\overline{\gamma}$ connecting $x$ and 
$x'$ passing through $z$. Since $z$ is in $\Gamma(x,v)$ we also have $z \in \tilde{\gamma} = T_s \gamma(x,v)$ 
for some $s \in G$. We have $\tilde{\gamma} = \overline{\gamma}$ because if $\overline{\gamma}$
would be shorter, the curve segment $T_{-s} \overline{\gamma}$ (which has the same length as 
$T_t$ are isometries) would also be a shorter connection than $\gamma$ connecting $x$ with $z$.
\end{proof}

\paragraph{}
{\bf Definition.} Given two tubes $\Gamma(x,v)$ and $\Gamma(y,w)$ with $x \neq y$.
The {\bf intersection number} $I(x,v,y,w)$ is the number of circular intersections of $\Gamma(x,v)$ 
and $\Gamma(y,w)$. Each is an orbit of $G=\mathbb{T}^1$ and so homotopically nontrivial in the cylinder
$Z(x,v) = \Gamma(x,v) \setminus \{x,x'\}$ and $Z(y,w)  = \Gamma(y,w) \setminus \{y,y'\}$ meaning that they can not 
be continuously deformed to a point within $Z(x,v)$ or $Z(y,w)$. 

\paragraph{}
The intersection number $I$ is almost everywhere constant: 

\begin{lemma}
$I(x,v,y,w)$ is constant if $x \neq y$ or $x' \neq y'$ and $v \notin T_xN$ and $w \notin T_yN$.
\end{lemma}
\begin{proof}
If we change $x,y,v,w$, the circles change smoothly and especially their radius. There are 
three ways to change the number of circles, (i) through bifurcation, 
Or then (ii) by slipping off through the end, where the circle radius can become zero.
or then (iii) when $v$ or $w$ go towards $T_xN$ or $T_yN$, where tubes become geodesics.
Case (i) is impossible because two circles must have distance larger than the radius of injectivity. 
Case (ii) was excluded by avoiding $x=y$ or $x'=y'$ and (iii) is excluded by the condition
on the vectors $v,w$. 
\end{proof}

\paragraph{}
The same argument applies in the almost connected case as the circles must be bounded
away from the additional single fixed point $\{p\}$. 
For $M \in \mathcal{K}$ with connected or almost connected $N$, there 
is an integer intersection number $I$ which only depends on $M$. We do not yet know
yet how large the set is where $x=y,x' \neq y'$. It could be a sub-manifold. 

\paragraph{}
Given a point $x \in N$, let $x_n \in M \setminus N$ be a sequence of points converging
to $x$. The radius of the circles $G x_n$ goes to zero. Their osculating circles define
$2$-planes $\Sigma(x_n)$ in $T_{x_n}M$ which converge to a $2$-plane $\Sigma(x)$ in $T_xM$. 

\begin{lemma}[A planar bundle]
If $G=\mathbb{T}^1$ acts effectively on $M \in \mathcal{K}$, then the fixed point set $N$
carries a natural complex line bundle $L(x)$ defined by the velocity $F(x)$ and acceleration vectors $G(x)$ 
of the circles $Gy$ for points $x$ close to $N$.
\end{lemma}
\begin{proof}
The velocity field $F$ of the circle action is also called a {\bf Killing vector field}.
For a circle $Gx$ parametrized by arc length the velocity vector $\vec{r}'(t)$ is $F/|F|$.
The circles get smaller and smaller near $N$ so that their curvature gets bigger. Sufficiently
close to $N$, the acceleration $\vec{r}''(t)$ can not have zero length, so that we have a normal vector 
$G(x)$ which is perpendicular to $F(x)$. These two vectors $F(x),G(x)$ define an osculating plane $\Sigma(x)$.
If one identifies this plane with the complex plane, one can see $x \to \Sigma(x)$ as a 
complex line bundle. To summarize, a circular effective action on a positive curvature manifold
defines a complex line bundle on the fixed point set. 
\end{proof}

\paragraph{}
Given $x,y \in N$ and unit vectors $v \in T_xM \setminus T_xN$ and
$v \in T_yM \setminus T_yN$, then $\Gamma(x,v)$ and $\Gamma(y,w)$ 
and end points $x',y'$ that intersect in a collection of $I$ circles that are bounded
away from $N$. If $x=y$, and $v,w$ are in the 2-dimensional normal bundle, then $x'=x''$. 

\paragraph{}
In the co-dimension-$2$ case,
the vectors in $T_xM \setminus T_xN$ form an open ball in which the
group action produces circles of starting vectors. In a two ball, two circles intersect
if they are close enough. 

\begin{lemma}[Codimension-2 case]
If $N$ has co-dimension $2$, then $\Gamma(x,v)$ and $\Gamma(x,w)$ must agree if $v$
and $w$ are close enough.
\end{lemma}

\begin{proof}
The intersection number $I(x,v,w) = I(x,x,v,w)$ is constant, as long
as the end points $y,z$ of $\Gamma(x,v),\Gamma(x,w)$ are not the same. 
There are three possibilities which could depend on $v,w$ and $x$. \\
Case A: $I(x,x,v,w)=n \neq 0$, \\
Case B: $I(x,x,v,w)=0$, meaning $\Gamma(x,v), \Gamma(x,w)$ can only intersect in $N$. \\
Case C: $I(x,x,v,w)= \infty$, meaning $\Gamma(x,v) = \Gamma(x,w)$.  \\
Case $A)$ can not happen because if $\Gamma(x,v)$ and $\Gamma(x,w)$ agree in a circle, then 
they must be the same tubes. In case $B)$, we prove that $y=z$ and so case C): 
Proof: Fix $v$.  Let $y \in N$ be the end point of $\Gamma(x,v)$ and $z \in N$ the 
end point of $\Gamma(x,w)$. Chose a path $z(t)$ of end points of $\Gamma(x,w(t))$
such that $z(t) \to y$ for $t \to 0$. As we are in co-dimension $2$, we must have 
an intersection if $z(t)$ is close enough to $y$. 
\end{proof}

\paragraph{} Having established that the tubes $\Gamma(x,v),\Gamma(x,w)$ all land in the
same point $x' \in N$, we have a {\bf return map} $\psi:N \to N$. As the tubes are made
of geodesics, they all have the same length $d(x,x')$ which is the return time. 
Now the story splits on whether we are in even or odd dimensions. In even dimension, 
$N$ can be a point or two points and $\psi = \psi^{-1}$ so that $\psi$ is an involution.
In the odd dimensional case, the lowest dimensional situation can be a bouquet of circles
with $m$ circles with fundamental group $\mathbb{Z}_m$.

\begin{coro}
The map $\psi:N \to N$ with $\psi(x)=x'$ is a smooth isometry of $N$.
\end{coro}
\begin{proof}
For a given $x \in N$, there is, independent of $v$, the same
end point $x' \in N$ so that the map $\psi$ is defined. As $\psi^{-1}=\psi$ we have 
$\psi^2=1$. The map $\psi$ is its own inverse. The map is also smooth. 
It is an isometry because if $\psi$ maps
geodesics into geodesics: if $a,b \in N$ are close enough, there is a unique geodesic 
$\gamma$ in $N$ connecting $a,b$. Now $\psi(\gamma)= \psi^{-1}(\gamma)$ is
a geodesic curve connecting $a'=\psi(a),b'=\psi(b)$. A linearization of $\psi$ shows
that if $\psi$ would expand length in the direction from $a$ to $b$, then $\psi^{-1}$
would shrink length. One can also look at the induced map on symmetric $2$-tensors $g$.
As the map is an involution, the eigenvalues of the linearization are $1$ or $-1$. 
But as positive definite metric tensors are mapped into itself, only the eigenvalues 
$1$ can occur.
\end{proof}

\paragraph{}
The fact that the return time $T$ is independent of $v$ allows us also to look at the 
limiting case, where the geodesics stays in $N$. The proof implies that also
within $N$, every geodesic has the same period $T$ or $2T$ in $N$.
If we look at the universal cover, then $\gamma_{nT}(x,v)=x$ for some integer 
$n=|\pi_1(M)|$. 

\begin{coro}[All geodesics are periodic]
If $M \in \mathcal{K}$ and $N=\phi(N)$ is connected, then there is a time 
$T$ which only depends on $M$ such that all geodesic orbits in $M$ have period $T$.
\end{coro}
\begin{proof}
We use the continuity of the flow. Given $x \in N$ and $v \in T_xN$, 
we get a geodesic $\gamma_t(x,v)$ in $N$. Let $v_k \to v$ be a sequence of 
unit vectors in $T_xM \setminus T_xN$. This produces a sequence of paths 
$\gamma_t(x,v_k)$ in $M \setminus N$ which all have period $T$ meaning
$\gamma_T(x,v_k)=x$. By continuity, we also have $\gamma_T(x,v)=x$.
This fixes all geodesics which start in $N$. Let $z$ now be a point in $M \setminus N$
and $w$ a unit vector in $T_zM$. This geodesic comes in a family
$\Gamma(z,w) = T_t \gamma(z,w)$ for which two points $x,x'$ are in $N$, hitting there
at unit vectors $v,v'$. Because $\gamma(z,w) = \gamma(x,v)$ and $\gamma(x,v)$ has period
$T$ also $\gamma(z,w)$ has period $T$.
\end{proof}

\paragraph{}
{\rm Remark.} There is a  $2$-sphere with a rotationally symmetric metric, where the round sphere or
Zoll sphere are all geodesics periodic \cite{GromollGrove1981, Zoll1903}.

\begin{coro}
For $0<t<T$, the geodesic ball $B_t(x) = \exp_t(B_1(x))$ is a ball in $M$.
\end{coro}

\begin{proof}
This is because the wave front $S_r(x)$ has not yet reached a
conjugate point. This is equivalent to the fact that $S_r(x)$ is
a sphere and $B_r(x)$ is a topological ball.
\end{proof}

\begin{coro}
If $\psi$ has period 2, then $M$ is a sphere. \\
If $\psi$ has period 1, then $M$ is a real projective space.
If $\psi$ has period $n>2$, then the universal cover of $N$ is a sphere and $N$ is a space form. 
\end{coro}

\begin{proof}
If $\psi$ has period $2$, the manifold $M$ is covered by two balls
$B_r(x), B_r(x')$ with $r>T(x,x')/2$. It must therefore be a topological
ball. If $\psi$ has period $1$, it can not be simply connected. Look
at the universal cover. There, the period is $2$ and the previous situation
applies. Now $M = S^{2d}/\mathbb{Z}_2 = \mathbb{RP}^{2d}$. 
\end{proof} 

\section{Caustic}

\paragraph{}
The set $C(x)$ of all points $y$ which are conjugate to $x$
in $M$ is also called the {\bf caustic} of $x$. This is in general different from the {\bf cut locus}
of $x$ but in the current situation with $x \in N$ connecting to $x' \in N$, the point $x'$ is both in
the caustic as well as in the cut locus. What we need for a caustic is a family of geodesics starting with $x$
ending up at $y$. If the Killing vector field $F$ defined
by the circle action is restricted to the geodesic it gives a Jacobi field $J(t)$
solving the Jacobi equation which is zero on $N$.

\paragraph{}
Caustics $C(x)$ can be complicated even for $2$-manifolds and even non-differentiable at many
places for differentiable $M$ \cite{elemente98}. Here with a group of isometries acting on $M$ and
connected fixed point set $N$ we are in a special situation, where a smooth geodesic sub-manifold $N$
of the manifold $M$ is the caustic $C(x)$ of each of its points.

\begin{coro}
If $M \in \mathcal{K}$ with connected or almost connected fixed point set $N$, 
then the manifold $N$ is the caustic for every $x \in N$.
\end{coro}
\begin{proof}
We see that $x'=\psi(x)$ is part of the caustic $C(x)$ because it is a conjugate point
as there is a one-parameter family of geodesics connecting $x$ with $x'$.
Because $\psi$ is an involution it is invertible and every point in $N$ is in the caustic of $x$.
As $\phi(N)$ has only one component, $N$ is the caustic of each of its
points $x \in N$. If $y$ is a point outside $N$, then it can not be in the caustic $C(x)$ as
if there were two geodesics connecting $x$ with $y$, then we had two tubes
$\Gamma(x,y),\Gamma(x,y)$ which both contain $x$ and the circle $Gy$.
\end{proof}

\paragraph{}
{\bf Remark.} In general, caustics are non-smooth and have cusps. In our case, the caustic $N$ is a smooth manifold.
That caustics are a mystery even in simplest situations is illustrated by the unsolved
Jacobi's last statement which claims that the caustics $C(x)$ of
a $x$ on a $2$-ellipsoid has exactly four cusps except if $x$ is an umbillical point.
The fact that the return time $T(x,v)$ is independent of $v$ can be seen geometrically:
The union $\Gamma(x) = \bigcup_{v \in T_xM \setminus T_xN} \Gamma(x,v)$ of all tubes
and the {\bf wave front} $S_r(x) = \{ y \in M \setminus N \; | d(x,y)=r \}$
which is the {\bf geodesic sphere} of radius $r$. They reach the point $x'$ at the same time.

\section{The almost connected case}

\paragraph{}
We only sketch the almost connected case. It just means that the geodesics pass
through an additional focal point before returning. That changes the length of the
minimal geodesics and is the reason for different pinching curvature constants. 
Let $N$ be the connected component and $\{ p \}$ the additional point.
There are now two possibilities. Either $\phi(M) = \{p,q\}$ consists of
two points or then, $\phi(M) = N \cup \{p\}$ where $N$ has positive dimension. 
What happens now is that we have the same situation as before just that the
sphere tube $\Gamma(x,v)$ starting at $N$ will focus on $p$ before returning
to $N$. We can then look at the map $\psi:N \to N$ as before. The only thing
we have to worry about is that only part of the geodesic could reach first $p$ while
other parts directly come back to $N$. 

\paragraph{}
{\bf Definition}. Given $x \in N, v \in T_xM \setminus T_xN$, let $J(x,v)$
denote the number of visits of $p$ before returning to $N$. 

\paragraph{}
Given $x \in N$, there are geodesics from $p$ to any point $x$ and so a 
sphere tube $\Gamma(x,p)$ 
The same analysis before shows that for all $x=p$ and all $v \in T_pM$, the
tubes $\Gamma(p,v), \Gamma(p,w)$ do not intersect if $v \neq w$. 

\begin{lemma}
For all $x \in N$, the geodesic sphere tube reaches first $p$, then returns 
back to $N$. 
\end{lemma}

\paragraph{}
We have now a complete picture of the geodesic flow. If $N=\{q\}$ and
$\phi(N)=\{p,q\}$ then the two charts $B_{3T/4}(p)$ and $B_{3T/4}(q)$
are balls which cover the entire manifold $M$ so that $M$ must be a sphere.
In the other case we proceed by induction. Assume $N=\mathbb{CP}^{d-1}$  with 
$d \geq 0$, then $N$ is covered by $d$ charts and $M$ is covered by $d+1$
charts, obtained by attaching a $2d$ dimensional ball $B_{3T/4}(p)$ 
to the charts of $N$ expanded by $B_{3T/4}$. This atlas covers $M$ 
and $M$ is topologically $\mathbb{CP}^d$.

\paragraph{}
The positive curvature manifolds
$\mathbb{RP}^{2d},\mathbb{S}^{2d}, \mathbb{CP}^{d}$ and $\mathbb{HP}^d$ all have connected or almost
connected fixed point components. As also the Wallach manifold 
$W^{12}$ can have a single fixed point component $W^6$ under a circular symmetry, a natural question is
whether all remaining positive curvature manifolds $\mathbb{W}^6,\mathbb{E}^6,W^{12},\mathbb{OP}^2,
\mathbb{W}^{24}$ can be constructed from earlier ones by steps using extensions $N \to M$                      
using $U(1),SU(2)$ as symmetries, the cases $2=\mathbb{S}^0 \to \mathbb{CP}^1=\mathbb{S}^2 
\to \mathbb{HP}^1=\mathbb{S}^4$ $\to \mathbb{OP}^1 = \mathbb{S}^8 \to_{+1}$ 
$\mathbb{OP}^2 \to_{+\mathbb{CP^2}}$ $\mathbb{HP}^5 \to \mathcal{W}^{24}$ getting from zero $0$ to $24$ 
dimensions $\mathcal{W}^{24}$ and whether the $SU(2)$ extension
$N=\mathbb{S}^2 + \mathbb{S}^2 + \mathbb{S}^0$ to $\mathbb{W}^6$ or $\mathbb{E}^6$ are the only
cases, where the fixed point set can have more than 2 components. It is a special transition as the Euler
characteristic, there does not either remain or increase by $1$ but gets to $6$. That
$W^6$ or $E^6$ can not be obtained from two copies of $\mathbb{CP}^2$ follows from Frankel's theorem
as they must then be connected. Frankel's theorem disallows extensions increasing the dimension by $4$
when going from $N$ to $M$ if one or more components have dimension $4$. So, 
the extension from dimension $2$ to dimension $6$ was the only one.

\section{Illustrations}

\begin{figure}[!htpb]
\scalebox{0.4}{\includegraphics{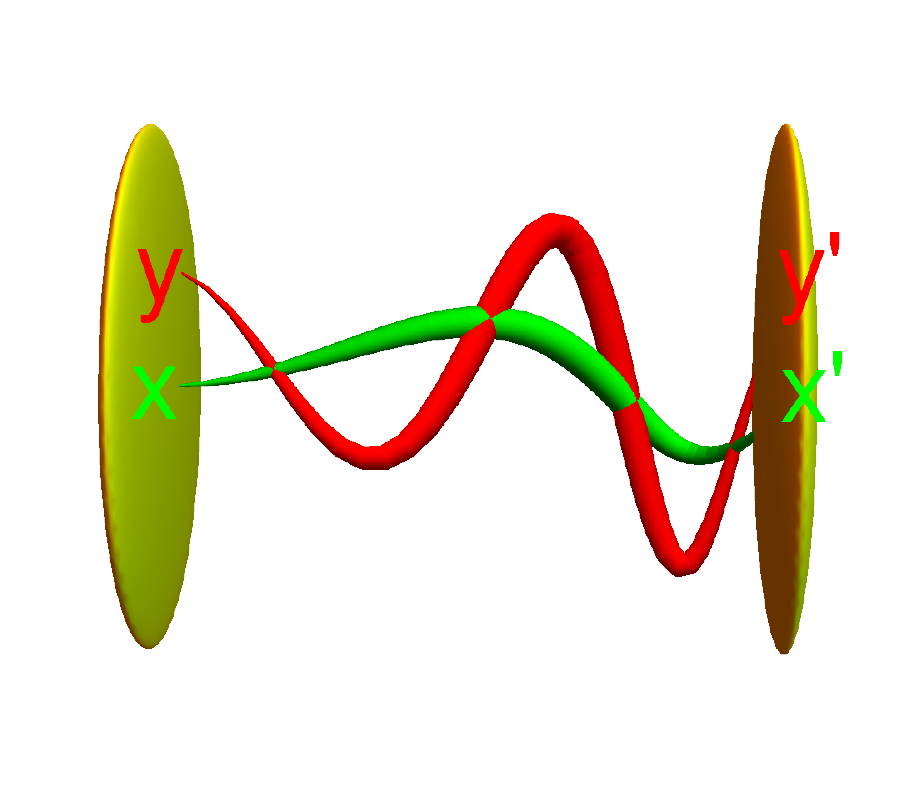}}
\scalebox{0.4}{\includegraphics{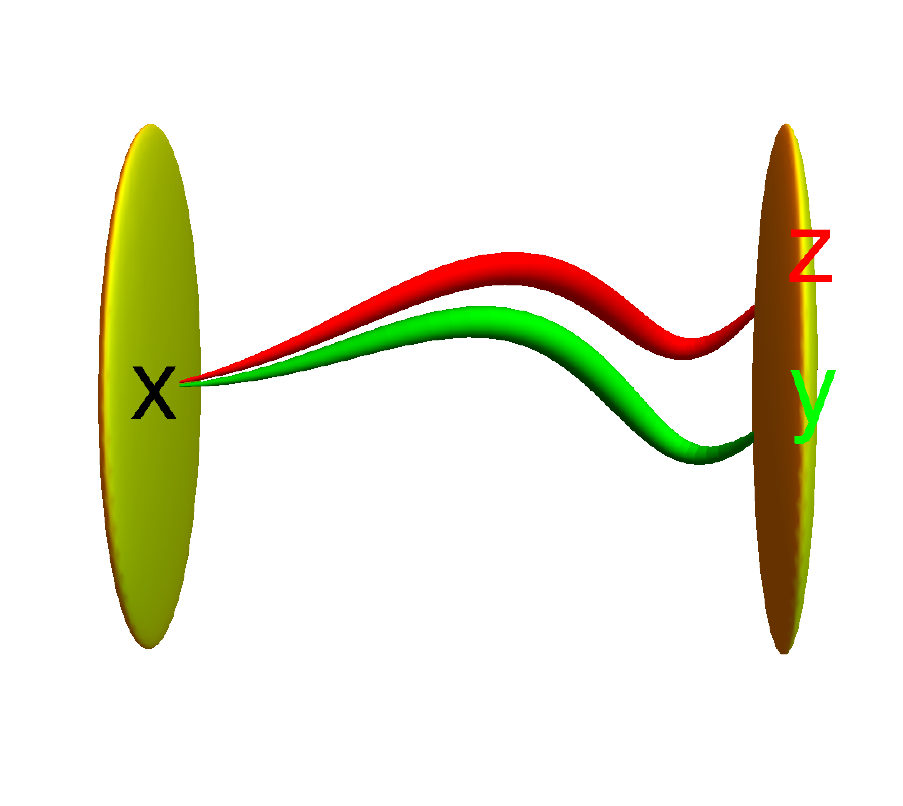}}
\scalebox{0.4}{\includegraphics{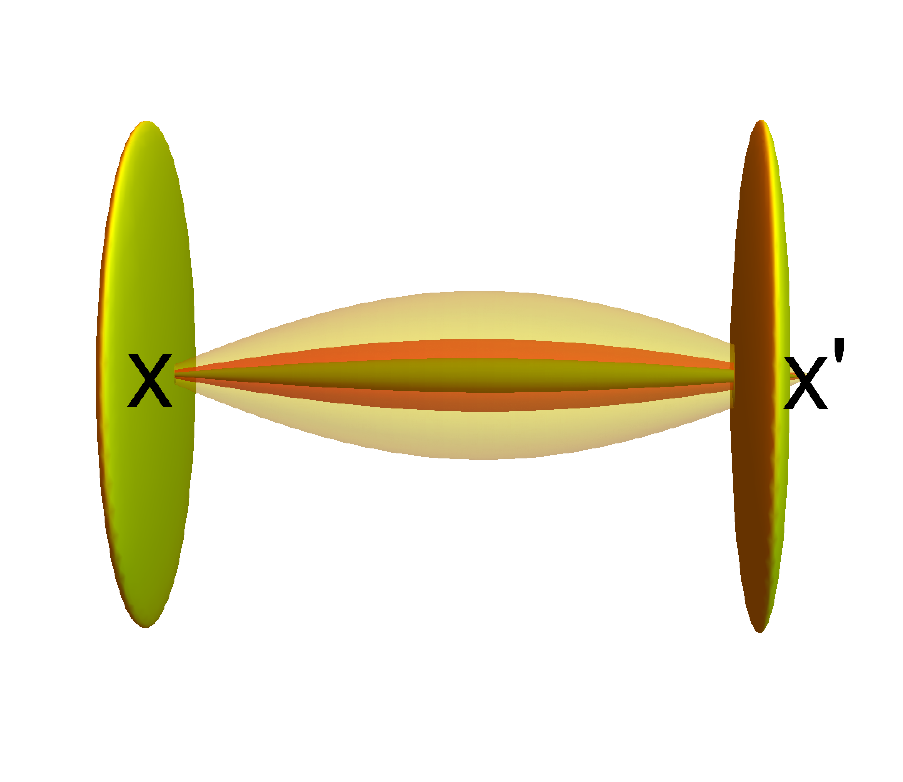}}
\caption{
Two spherical tubes $\Gamma(x,v), \Gamma(y,w)$ starting at $x \neq y$ and 
ending at $x' \neq y'$ intersect in a constant number $I(x,y,v,w)$ of circles. 
If $x=y$, then $x'=y'$. The tubes then either agree or do not intersect.
}
\end{figure}

\begin{figure}[!htpb]
\scalebox{0.45}{\includegraphics{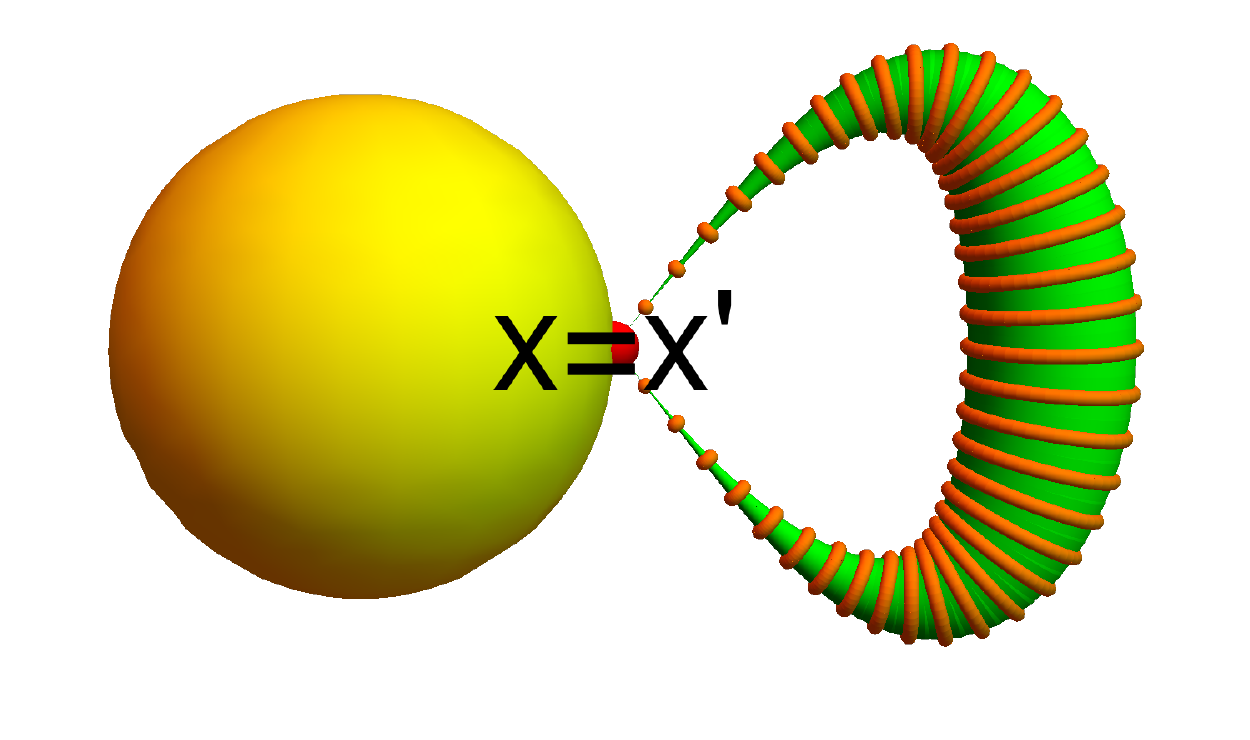}}
\scalebox{0.32}{\includegraphics{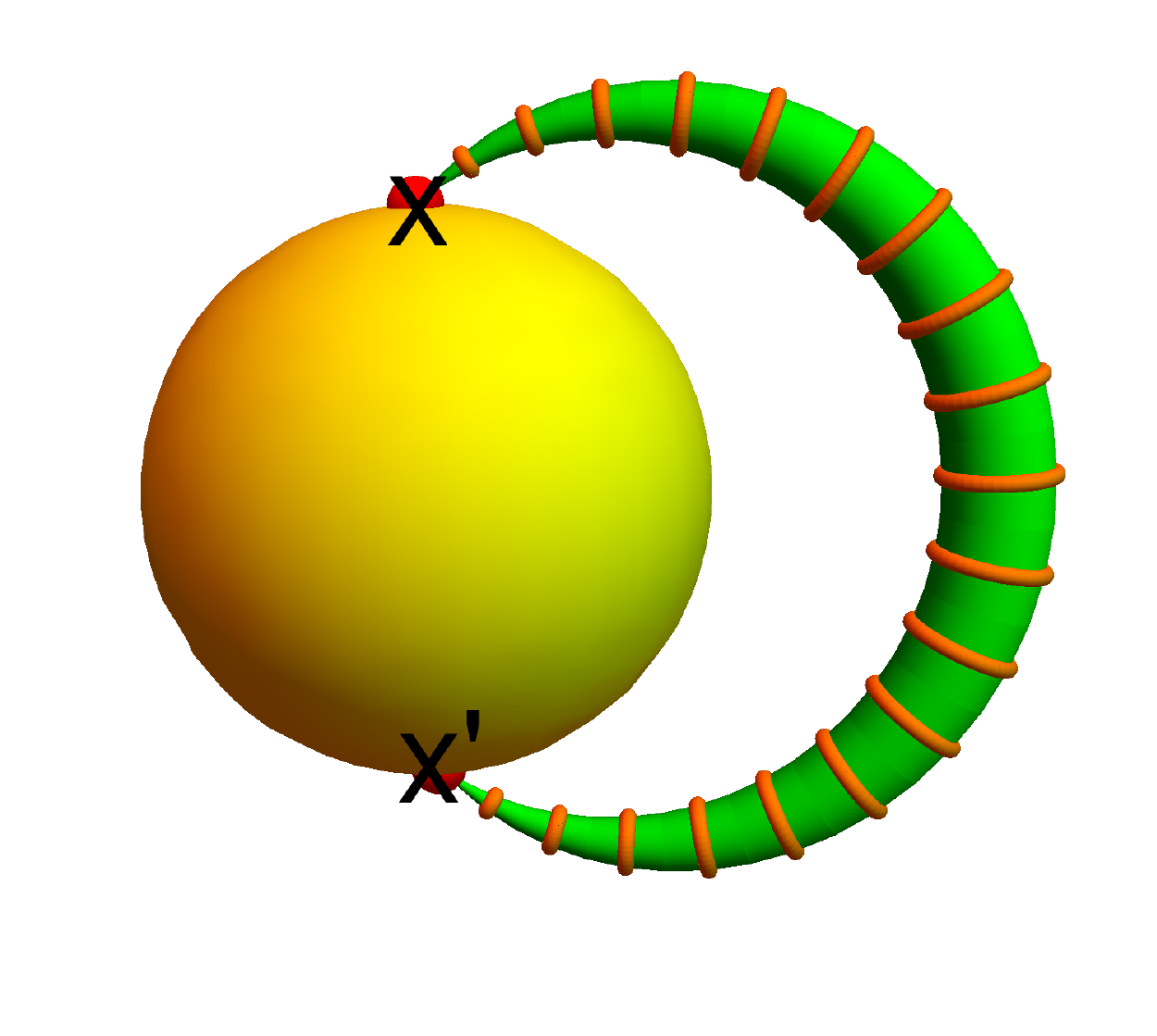}}
\scalebox{0.28}{\includegraphics{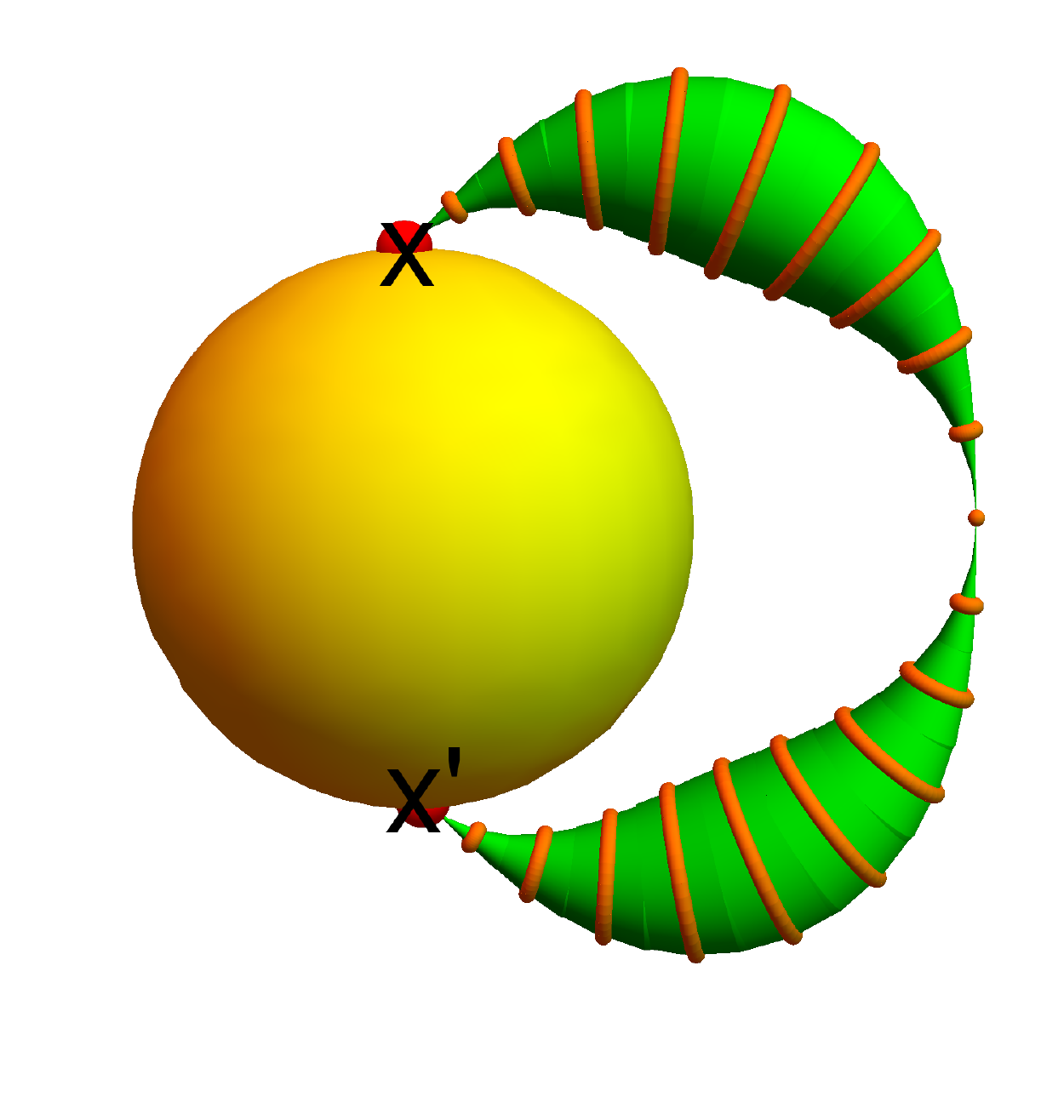}}
\caption{
The three cases of an extension from $N$ to $M$ for even-dimensional
manifolds of positive curvature.  They are 
$\mathbb{RP}^{2d}$, $\mathbb{S}^{2d}$ and $\mathbb{CP}^d$. 
}
\end{figure}

\begin{figure}[!htpb]
\scalebox{0.12}{\includegraphics{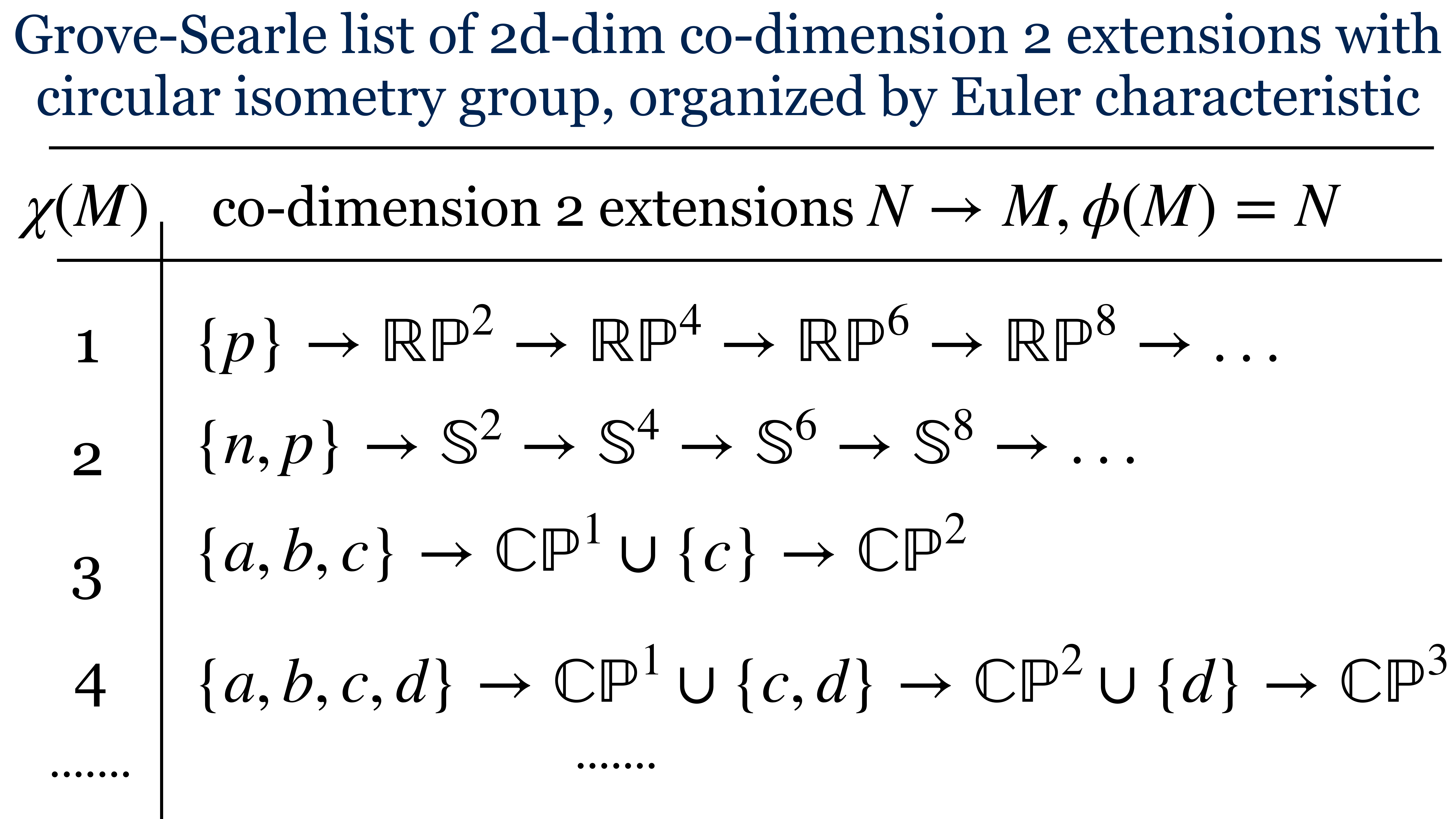}}
\scalebox{0.12}{\includegraphics{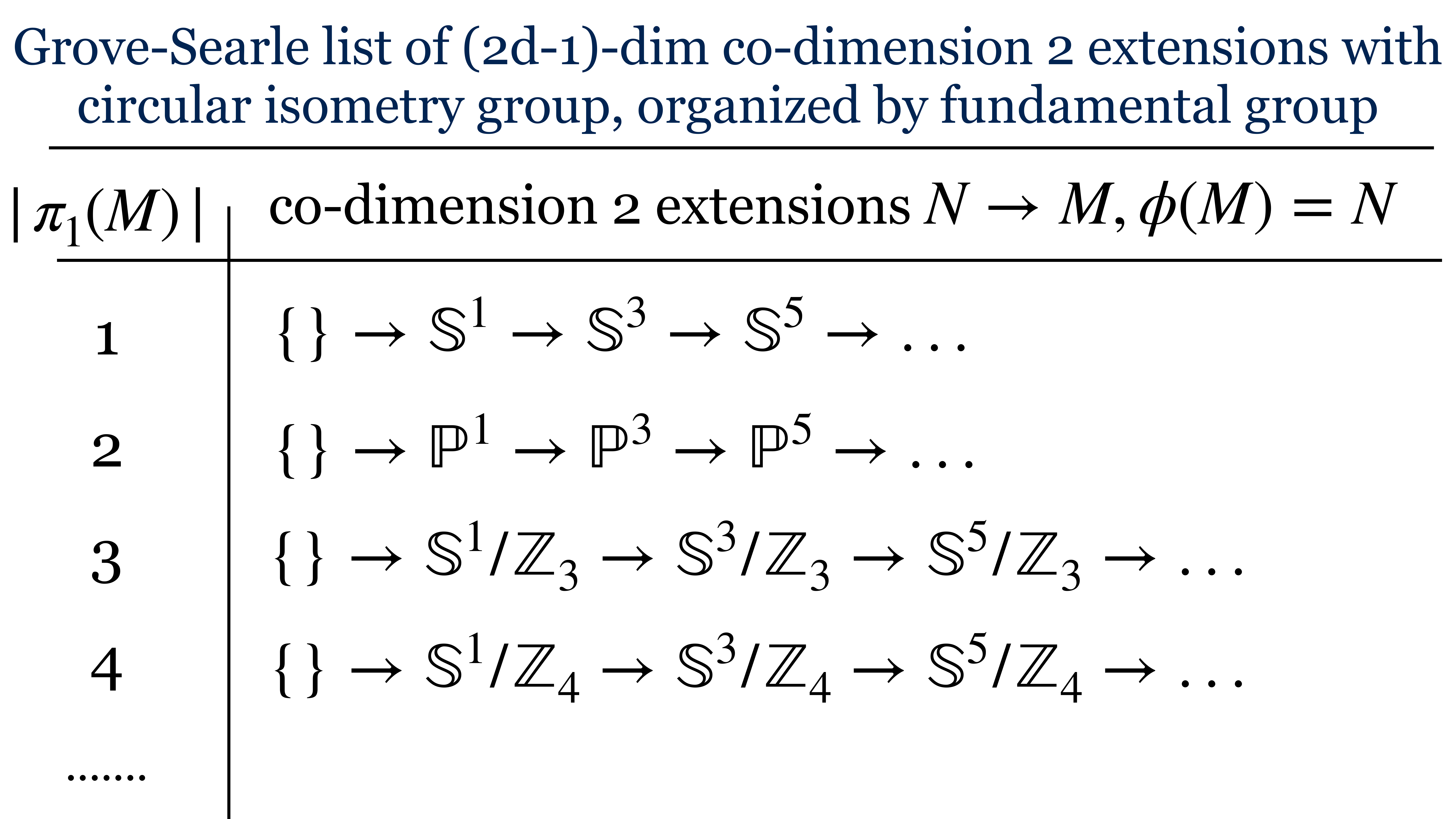}}
\label{galois1}
\caption{
$N=\phi(M)$ means that $N$ is the fixed point set
of a circular group $G$ on a connected manifold $M$.
}
\end{figure}

\begin{figure}[!htpb]
\scalebox{0.12}{\includegraphics{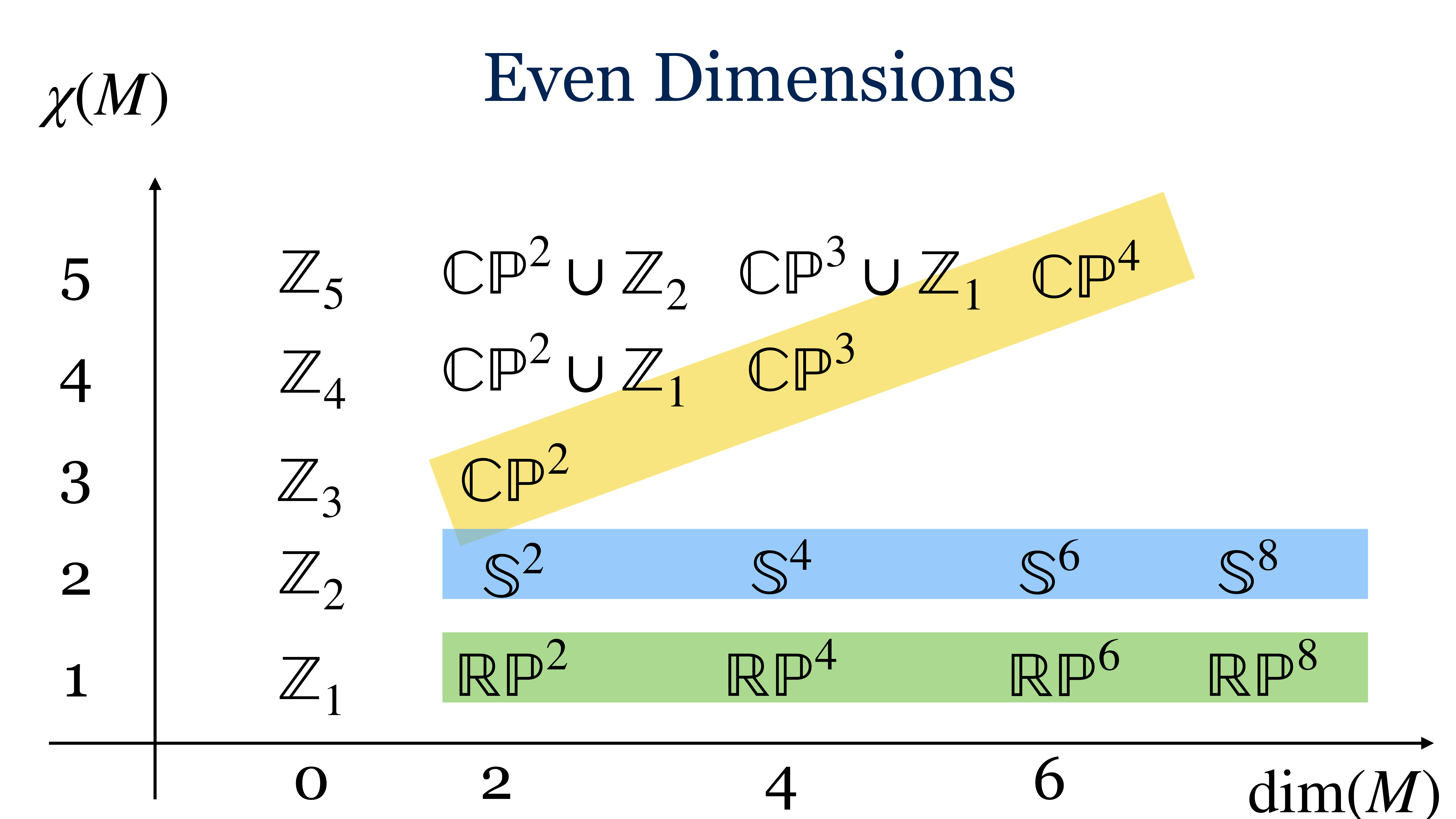}}
\scalebox{0.12}{\includegraphics{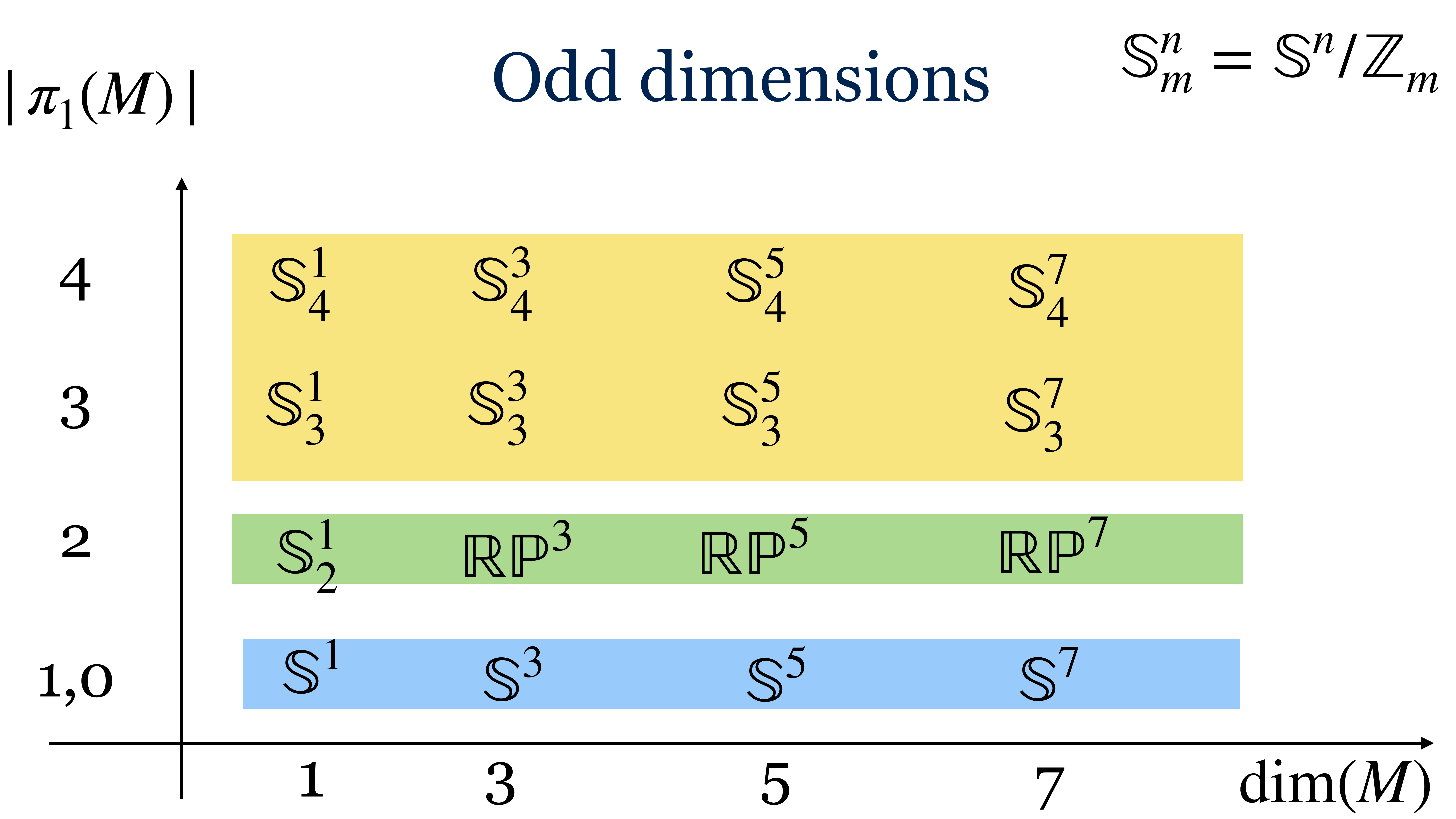}}
\label{galois2}
\caption{
An other ``periodic table" of the manifolds appearing in the
Grove-Searle theorem. The connected cases are high-lighted.
}
\end{figure}

\begin{figure}[!htpb]
\scalebox{0.12}{\includegraphics{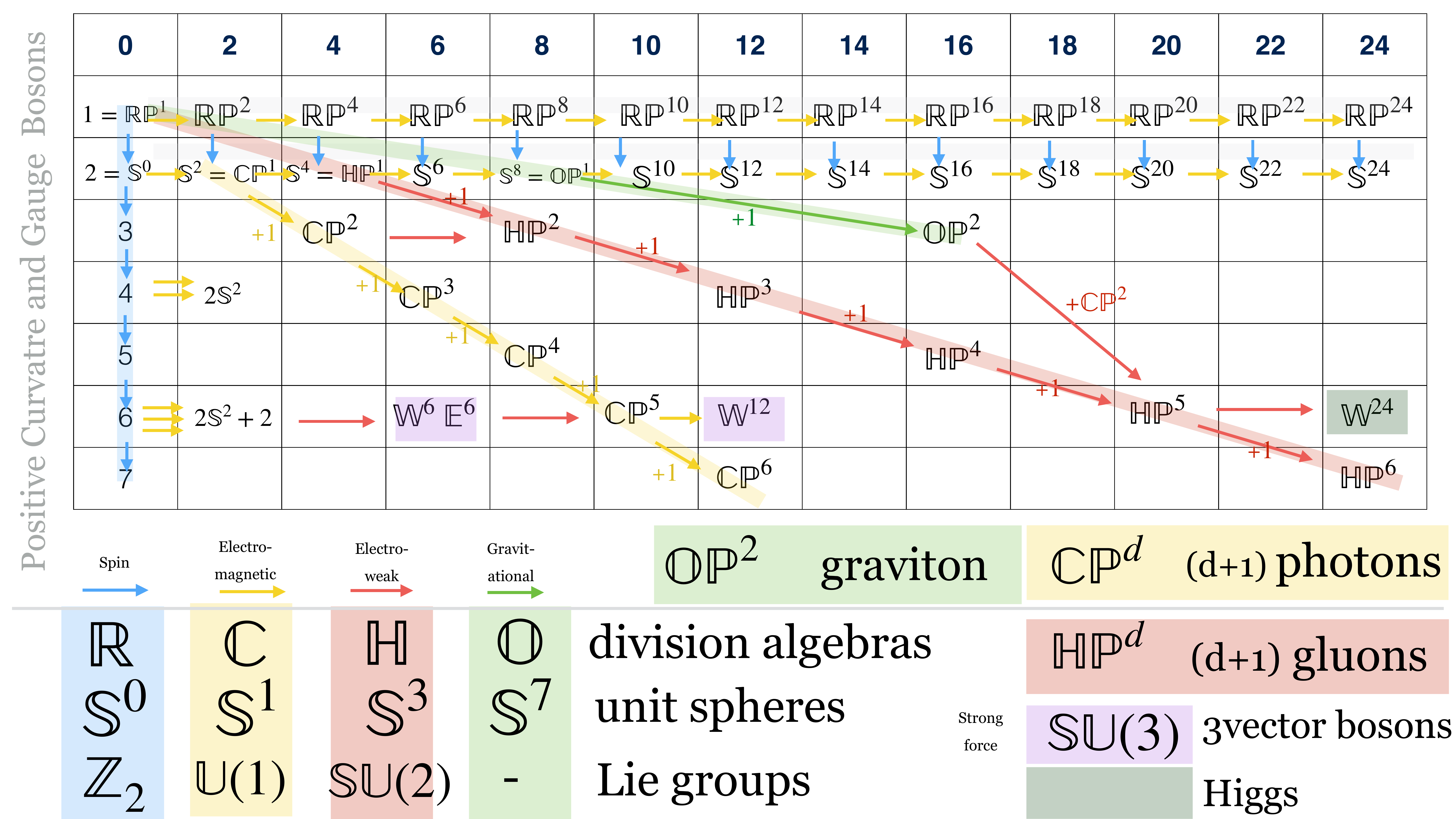}}
\scalebox{0.12}{\includegraphics{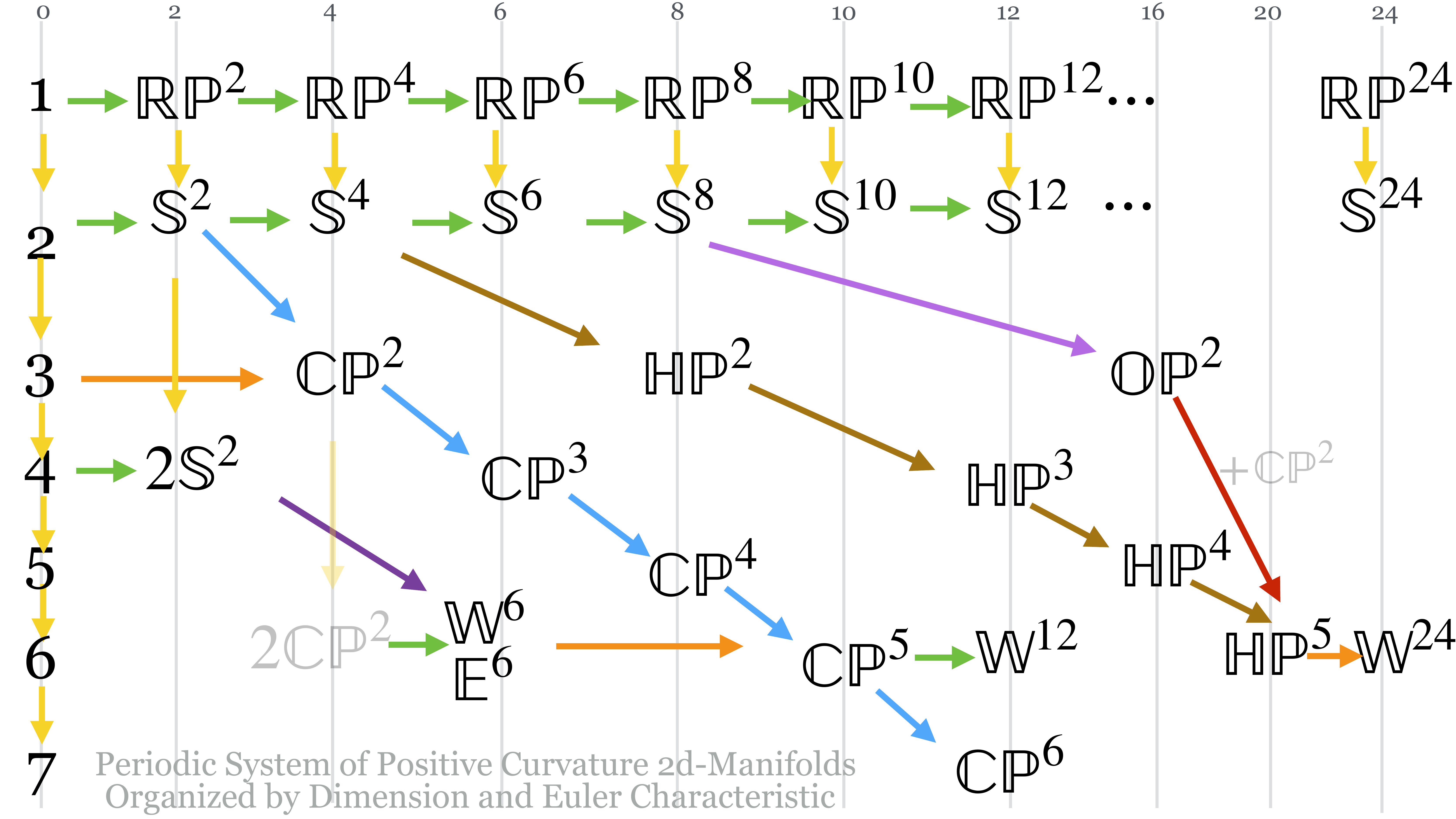}}
\label{galois2}
\caption{
Organizing all even dimensional
positive curvature manifolds in the ``Dimension-Euler characteristic" 
plane. 
}
\end{figure}

\section*{Appendix: Conner-Kobayashi-Lefschetz}

\paragraph{}
In full generality, for any Lie group $G$ action on $M$ by 
isometries on $M \in \mathcal{K}_{2d}$, the fixed point set $N$ is not
empty. This is a generalized version of a theorem of Berger \cite{Berger1961}. 
As the fixed point consists of smaller dimensional manifolds, one has 
the possibility to do reduction to smaller dimensional cases. 

\paragraph{}
The manifold $N$ can have different dimensions but it is {\bf totally geodesic} 
in $M$ meaning that any geodesic within $N$ stays in $N$. The proof is 
very simple: given two points $x,y \in N$ closer than the radius of injectivity
of $M$ (which is positive and finite for positive curvature manifolds), then
there exists a unique geodesic $\gamma$ in $M$ connecting $x$ with $y$. Assume there
is a point $z \in \gamma$ that is not in $M$, then the orbit $Gz$ is outside of $N$
and $G\gamma$ consists of a family of geodesics all having the same length connecting
$x$ with $y$. This contradicts that $x,y$ were closer than the radius of injectivity. 

\paragraph{}
Having a totally geodesic subset $N$ of $M$ assures that $N$ is a manifold.
See Theorem 5.1 in \cite{Kobayashi1972}. 
Proof: Take a point $x$ and a positive $r$ smaller than the radius of injectivity
of $M$. Then the set $U_r(x)=\exp_x(B_r(0))$ consisting of all geodesics of length $\leq r$
starting at $x$ is a chart of $M$. Then $U \cap N = \exp_x(U \cap V)$, where
$V$ is the fixed point set of $v \in T_xM$ under $G$. This shows that $U \cap F$ 
is a sub-manifold. Because every sectional curvature in $N$ is also a sectional
curvature in $M$, also each component of $N$ is a positive curvature manifold. 

\paragraph{}
The normal bundle of $N$ in $M$ can be made into a complex vector bundle (see
Theorem 5.3 in \cite{Kobayashi1972}), so that the codimension of each component
is even. Intuitively, positive curvature forces the Lie group to be odd dimensional:
Proof: let $G$ be an even dimensional Lie group. It has zero Euler characteristic
because there are vector fields on $G$ which have no equilibrium points (Poincar\'e-Hopf). 
But as the sectional curvatures are constant, it has to be a space form forcing it to 
be a sphere or projective space which both in even dimensions have positive Euler 
characteristic. As for $x \notin N$, the orbit $Gx$ must have positive curvature,
this forces the Lie group to be a sphere. There are only two cases $U(1)$ and $SU(2)$,
the $U(1)$ case because there are no curvature restrictions in one dimension. 

\paragraph{}
The fact that $\chi(M)=\chi(N)$ appears in Theorem 5.5 in \cite{Kobayashi1972}.
It is a simple fixed point theorem. Let me give a combinatorial proof using 
non-standard analysis \cite{Nelson77,NelsonSimplicity} and a discrete theorem
\cite{brouwergraph}. It takes the point of view that compact spaces can be 
modeled as finite graphs, where elements are connected if they have a distance
smaller than a given positive constant $h>0$. These graphs are not of manifold
type but homotopic to the graphs which are triangulations of $M$. 

\paragraph{}
If $(M,d)$ is a compact Riemannian manifold and $h>0$ is some small positive
infinitesimal (in the terminology of Nelsons IST, an extension of the ZFC
axiom system, this is defined and means that $h$ is a positive real number that is
smaller than any {\bf standard} positive number. The axioms IST clarify what standard means).
One can think of $h$  as a fixed Planck constant.
This defines a finite simple graph $G=(V,E)$, where $V$ is a finite set containing all standard
points in $M$ and where $(x,y)$ is in $E$ if $d(x,y)\leq \epsilon$. Now, this is a
finite simple graph. It is messy, as its dimension is huge. We do not mind however
as is homotopic to a triangulation of $M$ (homotopy is a well defined notion ported
from the continuum to the discrete. It has the same properties than the homotopy of
classical CW complexes). In particular the Whitney simplicial complex defined by $G$
has the same Betti numbers and Euler characteristic than $M$.

\paragraph{}
Now, if $T_t$ is a standard isometric circle action on $M$, then for every standard $t$
one has a {\bf graph auto-momorphism} of $G$ which is also an automorphim of the
finite abstract simplicial complex defined by it. The finite subset of $G$ which
are fixed by all $T_t$ with standard $t$, includes now of all standard points of $\phi(M)$ because
if a standard $x \in M$ satisfies $T_t(x)=x$ for all $t$, then $T_t(x)=x$ for
all standard $t$ and $x$ is a fixed point of the finite cyclic group containing
all the standard automorphisms of $G$. Now, the fact that $\phi(M)$ and $M$
have the same Euler characteristic is a consequence of a discrete fixed point theorem.
Lets remind about this result  \cite{brouwergraph}:

\paragraph{}
If an automorphism $T$ of a finite abstract simplicial complex $G$ is given,
then $T$ is just a simplicial map preserving the
incidence structure in $G$, then it induces a linear map $T_k$ on each cohomology
$H^k(G)$ (which are finite dimensional vector spaces given concretely as null-spaces
of matrices). The {\bf Lefschetz number} $L(G,T)$ is the super trace
$\sum_{k \geq 0} (-1)^k {\rm tr}(T_k)$. The index of a fixed point $x \in G$
is $i_T(x) = \omega(x) {\rm sign}(T|x)$, where ${\rm sign}(T|x)$ is the signature of the
permutation which $T$ induces on $x$ and $\omega(x) = (-1)^{\rm dim}(x)$.

\paragraph{}
The formula $\sum_{x \in G} i_T(x) = L(G,T)$ proven in \cite{brouwergraph} reduces to the definition
$\sum_{x \in G} \omega(x) = \chi(G)$ if $T$ is the identity.
The formula can be proven by heat deformation: if $U_T$ is the induced map
on forms (it is a matrix of the same size then $H$),
then $l(t)={\rm str}(\exp(-tL) U_T)$ is independent of $t$ by a discrete analog of a
theorem of McKean and Singer. But
$l(0) = {\rm str}(U_T)$ is $\sum_{T(x)=x} i_T(x)$ and  $\lim_{t \to \infty} l(t)=L(T,G)$ by the
Hodge theorem.

\paragraph{}
The Hodge theorem telling that $\chi(M) = \sum_{k \geq 0} {\rm dim}({\rm ker}(H_k)$, where $H_k$ is the block
in the Hodge Laplacian $H = D^2 = (d+d^*)^2 = d d^* + d^* d$ acting on $k$-forms
can be seen as the limiting case of the {\bf McKean-Singer heat deformation theorem}
${\rm str}(e^{-t H}) = \chi(G)$ for all $t$. In the case $t=0$, this is the
{\bf Euler-Poincar\'e theorem} for  Euler characteristic, in the case $t=+\infty$
this is the classical {\bf Hodge theorem} adapted to the discrete.
Now, both Euler-Poincar\'e and the Hodge theorem are elementary
linear algebra results in the discrete as we deal with finite matrices.

\paragraph{}
The McKean-Singer formula is understood by
seeing {\bf Dirac operator} $D=d+d^*$ induces a symmetry between even and
odd differential forms. It leads to ${\rm str}(H^k)=0$ for all $k>0$ implying
the {\bf super-symmetry statement} that the set of non-zero eigenvalues of $H$ restricted to
even forms is the same than the set of non-zero eigenvalues of $H$ restricted to
the odd forms. The {\bf heat flow washes the positive eigenvalues away}
because by nature of $H=D^2$, all eigenvalues $\lambda_k$ of $H$ are
non-negative and $e^{-t\lambda_k} \to 0$ for $t \to \infty$ if $\lambda_k>0$.
So, only the harmonic parts survive.

\bibliographystyle{plain}

\end{document}